\documentclass[a4paper,11pt]{article}
\usepackage{times, url}
\usepackage{amssymb}
\usepackage{amsmath}
\usepackage{amsthm}
\usepackage{amsfonts}
\usepackage{amscd}
\usepackage{graphicx}
\usepackage{cite}

\newtheorem{theorem}{Theorem}

\newtheorem{lemma}[theorem]{Lemma}

\begin{document}

\begin{center}
\vskip 1cm{\LARGE\bf
Hessenberg--Toeplitz Matrix Determinants with Schr\"{o}der and Fine Number Entries\\
\vskip .1in } \vskip 1cm
\end{center}

\large\begin{center}
Taras Goy\\
Faculty of Mathematics and Computer Science\\
Vasyl Stefanyk Precarpathian National University \\
Ivano-Frankivsk 76018\\
Ukraine\\
{\tt tarasgoy@pnu.edu.ua} \\
\vskip .2 in
Mark Shattuck\\
Department of Mathematics\\
University of Tennessee\\
Knoxville, TN 37996\\
USA\\
{\tt mark.shattuck2@gmail.com}\\
\end{center} \normalsize

\vskip .2 in

\begin{abstract}
In this paper, we find determinant formulas of several Hessenberg--Toeplitz matrices whose nonzero entries are derived from the small and large Schr\"{o}der and Fine number sequences.  Algebraic proofs of these results can be given which make use of Trudi's formula and the generating function of the associated sequence of determinants.  We also provide direct arguments of our results that utilize various counting techniques, among them sign-changing involutions, on combinatorial structures related to classes of lattice paths enumerated by the Schr\"{o}der and Fine numbers. As a consequence of our results, we obtain some new formulas for the Schr\"{o}der and Catalan numbers as well as for some additional sequences from the OEIS in terms of determinants of certain Hessenberg--Toeplitz matrices.
\end{abstract}

\section{Introduction}

Let $S_n$ denote the $n$-th \emph{large} Schr\"{o}der number given by
$$S_n=\frac{1}{n}\sum_{k=1}^n2^k\binom{n}{k}\binom{n}{k-1}, \qquad n \geq 1,$$
with $S_0=1$.  The \emph{small} Sch\"{o}der number $s_n$ is defined as $s_n=\frac{1}{2}S_n$ for $n \geq 1$, with $s_0=1$.  The $n$-th Fine number, denoted here by $t_n$, is given by
$$t_n=3\sum_{k=1}^{\lfloor\frac{n+1}{2}\rfloor}\binom{2n-2k}{n-1}-\binom{2n}{n}, \qquad n \geq 1,$$
with $t_0=0$.  The first several terms of the sequences $s_n$ and $t_n$ for $n \geq 0$ are as follows:
$$\{s_n\}_{n \geq 0}=\{1,1,3,11,45,197,903,4279,20793,103049,518859,\ldots\}$$
and
$$\{t_n\}_{n\geq0}=\{0,1,0,1,2,6,18,57,186,622,2120,\ldots\}.$$

We will make use of in our proofs the (ordinary) generating functions for $S_n$, $s_n$ and $t_n$, which are given respectively by
\begin{align*}
\sum_{n\geq0}S_nx^n&=\frac{1-x-\sqrt{1-6x+x^2}}{2x},\\
\sum_{n\geq0}s_nx^n&=\frac{1+x-\sqrt{1-6x+x^2}}{4x},\\
\sum_{n\geq0}t_nx^n&=\frac{1+2x-\sqrt{1-4x}}{2(2+x)}.
\end{align*}
Let $C_n=\frac{1}{n+1}\binom{2n}{n}$ denote the $n$-th Catalan number (see A000108 in \cite{Sl}) and recall
$$\sum_{n\geq0}C_nx^n=\frac{1-\sqrt{1-4x}}{2x}.$$  The preceding three sequences are closely aligned with $C_n$. For example,
we have for $n \geq 1$,
$$
S_n=\sum_{k=0}^{n}\binom{n+k}{2k}C_k \quad \text{ and } \quad
t_{n+1} = \frac{1}{2}\sum_{k=2}^{n}\frac{C_k}{(-2)^{n-k}}.
$$
Additional relations for the Fine numbers are given by $C_{n}=2t_{n+1}+t_{n}$ and
$$
t_{n+1} = C_n-\sum_{k=0}^{n-1} C_kt_{n-k}.
$$

The sequences $S_n$, $s_n$ and $t_n$ arise in various settings in enumerative and algebraic combinatorics and give the cardinalities of some important classes of first quadrant lattice paths \cite{Callan,Chen,Deng,Deutsch,DSh,Shapiro}.  See entries A006318, A001003 and A000957, respectively, in \cite{Sl} for further information.  Here, we will be interested in some new combinatorial properties of these numbers related to their occurrence in certain Hessenberg--Toeplitz matrices.

Many relations for Schr\"{o}der and Fine numbers have previously been found (see, e.g., \cite{DSh,Qi-KJM,Qi-Integers} and references contained therein), and determinants of matrices with Schr\"{o}der or Fine number entries and their generalizations have attracted recent attention.  A couple of basic results in this direction involve Hankel determinants for the Schr\"{o}der numbers, namely, $\det\!\big(S_{i+j}\big)_{i,j=0}^{n-1} =2^{\binom{n}{2}}$ and $\det\!\big(S_{i+j+1}\big)_{i,j=0}^{n-1}=2^{\binom{n+1}{2}}$.  The comparable formulas for the Fine numbers (see \cite{DSh}) are given by $\det\!\big(t_{i+j+1}\big)_{i,j=0}^{n-1} =1$ and $\det\!\big(t_{i+j+2}\big)_{i,j=0}^{n-1} =1-n$.  These results have been generalized in different ways by considering various families of Catalan-like sequences; see, e.g., \cite{El,MW} and reference contained therein.

In \cite{Qi}, Qi presents negativity results for a class of Toeplitz--Hessenberg determinants whose elements contain the products of the factorial and the large Schr\"{o}der numbers.  By using Cramer's rule together with a generating function approach, Deutsch \cite{Deutsch} obtained the following Fine--Catalan determinant relation
\begin{equation}\label{Deutsch}
t_n=(-1)^{n-1} \left|\begin{array}{ccccccc}
C_0        & 1      & 1   &  \cdots       & 0      & 0 \\
C_1        & C_0       & 1   & \cdots        & 0      & 0 \\
C_2        & C_1       & C_0      & \cdots        & 0      & 0 \\
\cdots     & \cdots    & \cdots & \ddots & \cdots & \cdots \\
C_{n-2}   & C_{n-3}  & C_{n-4} & \cdots  & C_0    & 1 \\
C_{n-1}     & C_{n-2}& C_{n-3}   & \cdots & C_1    & C_0
\end{array}\right|,
\end{equation}
which was later rediscovered by the authors in \cite{GSh-Cat} (formula (2.21) with $a=-1$). In \cite{GSh-Cat}, the authors found determinants of several families of Toeplitz--Hessenberg matrices having various subsequences of the Catalan sequence for the nonzero entries. These determinant formulas may  also be rewritten equivalently as identities involving sums of products of Catalan numbers and multinomial coefficients. Comparable results featuring combinatorial arguments have been found for the generalized Fibonacci (Horadam), tribonacci and tetranacci numbers in \cite{GSh-Hor,GSh-tri,GSh-tetr}.  See also \cite{BBD}, where further results for Horadam number permanents and determinants are obtained using an algebraic approach.

The organization of this paper is as follows.  In the next section, we find formulas providing algebraic arguments of several Hessenberg--Toeplitz matrices whose nonzero entries are given by $S_n$, $s_n$, $t_n$ or translates thereof. As a consequence of our results, one obtains new determinant expressions, and hence combinatorial interpretations, for $S_n$, $s_n$ and $C_n$, as well as for several additional sequences occurring in \cite{Sl}.  Further, equivalent multi-sum versions of these determinant identities may be obtained using Trudi's formula (see Lemma \ref{lem1} below).  In the third section, we provide combinatorial proofs of the preceding formulas upon making use of various counting techniques such as direct enumeration, bijections between equinumerous structures and sign-changing involutions.  In doing so, we draw upon the well-known combinatorial interpretations of $S_n$, $C_n$, $s_n$ and $t_n$ as enumerators of certain classes of first quadrant lattice paths.

\section{Schr\"{o}der and Fine number determinant formulas}

A Hessenberg--Toeplitz matrix is one having the form
\begin{equation}\label{DetHess}
A_n:= A_n(a_0; a_1,\ldots,a_n)=\left(\begin{array}{ccccccc}
a_1        & a_0      & 0   &  \cdots       & 0      & 0 \\
a_2        & a_1       & a_0   & \cdots        & 0      & 0 \\
a_3        & a_2       & a_1      & \cdots        & 0      & 0 \\
\cdots     & \cdots    & \cdots & \ddots & \cdots & \cdots \\
a_{n-1}   & a_{n-2}  & a_{n-3} & \cdots  & a_1    & a_0 \\
a_{n}     & a_{n-1}& a_{n-2}   & \cdots & a_2    & a_1
\end{array}\right),
\end{equation}
where $a_0\neq 0$.   The following multinomial expansion of $\det(A_n)$ in terms of a sum of products of the $a_i$ is known as \emph{Trudi's formula} (see, e.g., \cite[Theorem 1]{MM1}).

\begin{lemma}\label{lem1}
Let $n$ be a positive integer.  Then
\begin{equation}\label{Trudi}
\det(A_n)=\sum_{\widetilde{v}=n}(-a_0)^{n-|v|}{|v| \choose v_1,\ldots,v_n}a_1^{v_1}a_2^{v_2}\cdots a_{n}^{v_n},
\end{equation}
where
$$
{|v|\choose v_1,\ldots,v_n}=\frac{|v|!}{v_1!v_2!\cdots v_n!},\quad  \widetilde{v}=v_1+2v_2+\cdots+nv_n, \quad |v|=v_1+v_2+\cdots+v_n,\,\, v_i\geq 0.$$
Equivalently, we have
\begin{equation*}
\det(A_n) =\sum_{k=1}^n(-a_0)^{n-k}\sum\limits_{\begin{smallmatrix} i_1,\ldots,i_k\geq1\\ i_1+i_2+\cdots+i_{k}=n \end{smallmatrix}}a_{i_1}a_{i_2}\cdots a_{i_k}.
\end{equation*}
\end{lemma}

It is seen that the sum in \eqref{Trudi} may be regarded as being over the set of partitions of the positive integer $n$. The special case of Trudi when $a_0=1$ is known as \emph{Brioschi's formula} \cite{Muir}. Here, we will focus on some cases of $\det(A_n)$ when $a_0=\pm 1$.  For the sake of brevity, we denote $\det\big(A_n(\pm1; a_1, a_2,\ldots, a_n)\big)$ by $D_{\pm}(a_1, a_2,\ldots, a_n)$.

There is the following inversion theorem involving $a_i$ and the corresponding sequence of Hessenberg--Toeplitz determinants when $a_0=1$ (see \cite[Lemma~4]{KP}):

\begin{lemma}\label{lem2}
Let $(b_n)_{n\geq0}$ be defined by $b_n=\det(A_n)$ for $n\geq 1$, where $A_n$ is given by \eqref{DetHess} with $a_0=b_0=1$.  If $B_n$ denotes the Hessenberg--Toeplitz matrix associated with $b_0,\ldots,b_n$, then $a_n=\det(B_n)$ for $n \geq 1$.
\end{lemma}

We have the following determinant identity formulas involving the large and small Schr\"{o}der numbers.

\begin{theorem} \label{Theo1} We have
	\begin{align}
	D_{+}(S_1, S_2,\ldots,S_{n})&=(-1)^{n-1}S_{n-1}, \qquad n \geq 2,\label{Theo1e1}\\
	D_{-}(S_1,S_2,\ldots,S_{n})&=2\cdot A134425[n-1],\label{Theo1e2}\\
   D_{+}(S_0,S_1,\ldots,S_{n-1})&=(-1)^{n-1}s_{n-1},\label{Theo1e3}\\
   D_{-}(S_0,S_1,\ldots,S_{n-1})&=s_n, \qquad \label{Theo1e4}\\
	D_{+}(s_1,s_2,\ldots,s_n)&=(-1)^{n-1}S_{n-1}, \label{Theo1e5}\\
D_{-}(s_1,s_2,\ldots,s_{n})&=A225887[n-1],\label{Theo1e6}\\
   D_{+}(s_0,s_1,\ldots,s_{n-1})&=(-1)^{n-1}A114710[n-1],\label{Theo1e7}\\
   D_{-}(s_0,s_1,\ldots,s_{n-1})&=S_{n-1}, \qquad \label{Theo1e8}\\
   D_{+}(s_2,s_3,\ldots,s_{n+1})&=(-1)^{n-1}S_{n-1}, \qquad n \geq 2, \qquad \label{Theo1e9}
	\end{align}
where all formulas hold for $n \geq 1$ unless stated otherwise.
\end{theorem}

Making use of Lemma \ref{lem1} yields the following multinomial identities for the two kinds of Schr\"{o}der numbers.

\begin{theorem} \label{Theo2} We have
	\begin{align}
	\sum_{\widetilde{v}=n}(-1)^{|v|-1}{|v|\choose v_1,\ldots,v_n}S_1^{v_1}S_2^{v_2}\cdots S_{n}^{v_n}&=
	S_{n-1},\quad n \geq 2,\label{F1}\\
\sum_{\widetilde{s}=n}{|v|\choose v_1,\ldots,v_n}S_1^{v_1}S_2^{v_2}\cdots S_{n}^{v_n}&=2\cdot A134425[n-1],  \label{F2}\\
	\sum_{\widetilde{v}=n}(-1)^{|v|-1}{|v|\choose v_1,\ldots,v_n}S_0^{v_1}S_1^{v_2}\cdots S_{n-1}^{v_n}&=s_{n-1},\label{F3}\\
\sum_{\widetilde{v}=n}{|v|\choose v_1,\ldots,v_n}S_0^{v_1}S_1^{v_2}\cdots S_{n-1}^{v_n}&=s_n, \label{F4}\\
\sum_{\widetilde{v}=n}(-1)^{|v|-1}{|v|\choose v_1,\ldots,v_n}s_1^{v_1}s_2^{v_2}\cdots s_{n}^{v_n}&=
	S_{n-1},\label{F5}\\
\sum_{\widetilde{s}=n}{|v|\choose v_1,\ldots,v_n}s_1^{v_1}s_2^{v_2}\cdots s_{n}^{v_n}&=A225887[n],  \label{F6}\\
	\sum_{\widetilde{v}=n}(-1)^{|v|-1}{|v|\choose v_1,\ldots,v_n}s_0^{v_1}s_1^{v_2}\cdots s_{n-1}^{v_n}&=A114710[n-1],\label{F7}\\
\sum_{\widetilde{v}=n}{|v|\choose v_1,\ldots,v_n}s_0^{v_1}s_1^{v_2}\cdots s_{n-1}^{v_n}&=S_{n-1}, \label{F8}\\
\sum_{\widetilde{v}=n}(-1)^{|v|-1}{|v|\choose v_1,\ldots,v_n}s_2^{v_1}s_3^{v_2}\cdots s_{n+1}^{v_n}&=S_{n-1}, \quad n \geq 2,\label{F9}
	\end{align}
where all formulas hold for $n \geq 1$ unless stated otherwise.
	\end{theorem}

The identities in Theorems  \ref{Theo1} and \ref{Theo2} are seen to be equivalent by \eqref{Trudi}, so we need only prove the former.

\begin{proof}
Let $f(x)=\sum_{n\geq1}\det(A_n)x^n$, where $A_n$ is given by \eqref{DetHess}.  Then rewriting \eqref{Trudi} in terms of generating functions implies
$$f(x)=\sum_{n\geq 1}(-a_0x)^n\sum_{\widetilde{v}=n}
	 {|v|\choose v_1,\ldots,v_n}
	\left(-\frac{a_1}{a_0}\right)^{v_1}\cdots\left(-\frac{a_n}{a_0}\right)^{v_n}=\frac{g(x)}{1-g(x)},$$
where $g(x)=\sum_{i \geq 1}(-a_0)^{i-1}a_ix^i$.

We consider several cases on $a_i$. First let $a_i=S_i$ for $i \geq 1$.  Then
$$g(x)=\sum_{i \geq 1}(-a_0)^{i-1}S_ix^i=\frac{1+3a_0x-\sqrt{1+6a_0x+a_0^2x^2}}{2a_0^2x}.$$
If $a_0=1$, then
\begin{align*}
f(x)&=\frac{g(x)}{1-g(x)}=\frac{1+3x-\sqrt{1+6x+x^2}}{-1-x+\sqrt{1+6x+x^2}}=2x-\frac{1}{2}\left(1+3x-\sqrt{1+6x+x^2}\right)\\
&=2x+\sum_{n\geq 2}(-1)^{n-1}S_{n-1}x^n,
\end{align*}
which implies \eqref{Theo1e1}.  If $a_0=-1$, then
\begin{align*}
f(x)&=\frac{g(x)}{1-g(x)}=\frac{1-3x-\sqrt{1-6x+x^2}}{-1+5x+\sqrt{1-6x+x^2}}=\frac{4x}{1-7x+\sqrt{1-6x+x^2}}\\
&=\sum_{n\geq 1}2\cdot A134425[n-1]x^n,
\end{align*}
which implies \eqref{Theo1e2}, upon recalling the formula $\sum_{n\geq0}A134425[n]x^n=\frac{2}{1-7x+\sqrt{1-6x+x^2}}$ (see OEIS article).

Now let $a_i=S_{i-1}$ for $i \geq 1$.  In this case, we have
$$g(x)=\sum_{i\geq 1}(-a_0)^{i-1}S_{i-1}x^i=\frac{-1-a_0x+\sqrt{1+6a_0x+a_0^2x^2}}{2a_0}.$$
If $a_0=1$, then
$$f(x)=\frac{-1-x+\sqrt{1+6x+x^2}}{3+x-\sqrt{1+6x+x^2}}=\frac{-1+x+\sqrt{1+6x+x^2}}{4}=\sum_{n\geq1}(-1)^{n-1}s_{n-1}x^n,$$
which gives \eqref{Theo1e3}, whereas if $a_0=-1$, then
$$f(x)=\frac{1-x-\sqrt{1-6x+x^2}}{1+x+\sqrt{1-6x+x^2}}=\frac{1-3x-\sqrt{1-6x+x^2}}{4x}=\sum_{n\geq 1}s_nx^n,$$
which gives \eqref{Theo1e4}.

Similar proofs may be given for \eqref{Theo1e5}--\eqref{Theo1e8}.  Alternatively, formulas \eqref{Theo1e5} and \eqref{Theo1e8}  follow from \eqref{Theo1e4} and \eqref{Theo1e3}, respectively, upon applying Lemma \ref{lem2} since

	\begin{align*}
	&D_+(s_1,\ldots,s_n)=(-1)^{n-1}S_{n-1} \text{ if and only if }\\ &D_-(S_0,\ldots,S_{n-1})=D_+(S_0,-S_1,\ldots,(-1)^{n-1}S_{n-1})=s_n
	\end{align*}
	and
	\begin{align*}
	&D_+(S_0,\ldots,S_{n-1})=(-1)^{n-1}s_{n-1} \text{ if and only if }\\
	&D_-(s_0,\ldots,s_{n-1})=D_+(s_0,-s_1,\ldots,(-1)^{n-1}s_{n-1})=S_{n-1}.
	\end{align*}

Finally, to show \eqref{Theo1e9}, let $a_i=s_{i+1}$ for $i \geq 1$ and $a_0=1$ to get
$$g(x)=\sum_{i\geq1}(-1)^{i-1}s_{i+1}x^i=\frac{-1-3x+4x^2+\sqrt{1+6x+x^2}}{4x^2}.$$
Thus,
\begin{align*}
f(x)&=\frac{g(x)}{1-g(x)}=\frac{-1-3x+4x^2+\sqrt{1+6x+x^2}}{1+3x-\sqrt{1+6x+x^2}}\\
&=3x+\frac{1}{2}\left(-1-3x+\sqrt{1+6x+x^2}\right)=3x+\sum_{n\geq 2}(-1)^{n-1}S_{n-1}x^n,
\end{align*}
which completes the proof.
\end{proof}

We have the following Fine number determinant formulas.

\begin{theorem}\label{Theo3}
We have
\begin{align}
D_+(t_1,t_2,\ldots,t_{n})&=u_n,\label{Theo3e1}\\
D_-(t_1,t_2,\ldots,t_n)&=C_{n-1}, \label{Theo3e2}\\
D_+(t_2,t_3,\ldots,t_{n+1})&=(-1)^{n-1}C_{n-1}, \quad n \geq 2,\label{Theo3e3}\\
D_-(t_2,t_3,\ldots,t_{n+1})&=A137398[n],  \label{Theo3e4}\\
D_+(t_3,t_4,\ldots,t_{n+2})&=(-1)^{n-1}A030238[n-1],\label{Theo3e5}\\
D_+(t_4,t_5,\ldots,t_{n+3})&=(-1)^{n-1}C_{n-1}, \qquad n \geq 3, \label{Theo3e6}
\end{align}
where all formulas hold for $n \geq 1$ unless stated otherwise and $u_n$ denotes the sequence defined recursively by
$u_n=u_{n-1}+\sum_{i=1}^{n-2}(-1)^{i+1}C_iu_{n-i-1}$ if $n \geq 3$
with $u_1=u_2=1$.
\end{theorem}
\begin{proof}
Proofs comparable to those given for \eqref{Theo1e1}--\eqref{Theo1e9} may also be given for \eqref{Theo3e1}--\eqref{Theo3e6}.  We illustrate using formula \eqref{Theo3e6}.  First note that
$$\sum_{n\geq 1}t_{n+3}x^n=\sum_{n\geq 3}t_{n+1}x^{n-2}=\frac{1}{x^2}\left(\frac{2}{1+2x+\sqrt{1-4x}}-1-x^2\right),$$
and hence we have
$$g(x)=\sum_{n\geq 1}(-1)^{n-1}t_{n+3}x^n=-\frac{1}{x^2}\left(\frac{1+2x-x^2+2x^3-(1+x^2)\sqrt{1+4x}}{1-2x+\sqrt{1+4x}}\right).$$
This gives
\begin{align*}
\sum_{n\geq 3}\det(A_n)x^n&=\frac{g(x)}{1-g(x)}-\det(A_1)x-\det(A_2)x^2\\
&=\frac{-1-2x+x^2-2x^3+(1+x^2)\sqrt{1+4x}}{1+2x-\sqrt{1+4x}}-2x+2x^2\\
&=\frac{-1-4x-x^2+2x^3+(1+2x-x^2)\sqrt{1+4x}}{1+2x-\sqrt{1+4x}}\\
&=\frac{\left(-1-4x-x^2+2x^3+(1+2x-x^2)\sqrt{1+4x}\right)\left(1+2x+\sqrt{1+4x}\right)}{4x^2}\\
&=\frac{-2x^2\left(1+2x-2x^2-\sqrt{1+4x}\right)}{4x^2}=x\left(\frac{1-\sqrt{1+4x}}{-2x}-1+x\right)\\
&=x\sum_{n\geq 2}C_n(-x)^n=\sum_{n\geq 3}(-1)^{n-1}C_{n-1}x^n,
\end{align*}
which implies \eqref{Theo3e6}.
\end{proof}

\section{Combinatorial proofs}

In this section, we provide combinatorial proofs of formulas \eqref{Theo1e1}--\eqref{Theo1e9} and \eqref{Theo3e1}--\eqref{Theo3e6}.  Let us first recall combinatorial interpretations of the sequences $S_n$, $s_n$ and $t_n$ which we will make use of in our proofs and define some related terms.  Let $\mathcal{P}_n$ denote the set of lattice paths (called \emph{Schr\"{o}der} paths) from the origin to the point $(2n,0)$ that never go below the $x$-axis using $u=(1,1)$, $d=(1,-1)$ and $h=(2,0)$ steps.  Then $S_n=|\mathcal{P}_n|$ for all $n \geq 0$, where $\mathcal{P}_0$ is understood to consist of the empty path of length zero.  Half the horizontal distance traversed by a Schr\"{o}der path $\lambda$ will be referred to here as the \emph{length} of $\lambda$ and is denoted by $|\lambda|$.  Note that $|\lambda|$ equals the sum of numbers of $u$ and $h$ steps in $\lambda$. (We remark that the term \emph{semi-length} is often used in the literature, instead of length, for the quantity indicated, though we prefer the latter due to brevity.)  An $h$ step connecting two points with $y$-coordinate $\ell\geq 0$ is said to be of \emph{height} $\ell$.  A \emph{low} $h$ step will refer to an $h$ step of height $0$. The subset of $\mathcal{P}_n$ whose members contain no low $h$ steps will be denoted by $\mathcal{Q}_n$, with its members referred to as \emph{restricted} Schr\"{o}der paths. Then it is well-known that $s_n=|\mathcal{Q}_n|$ for $n \geq 0$.  Hence, since $S_n=2s_n$ if $n \geq 1$, we have that exactly half the members of $\mathcal{P}_n$ are restricted.

Let $\mathcal{D}_n$ denote the subset of $\mathcal{P}_n$ whose members contain no $h$ steps. Members of $\mathcal{D}_n$ are referred to as \emph{Dyck} paths, with $|\mathcal{D}_n|=C_n$ for $n\geq0$.  A member of $\mathcal{D}_n$ is said to have a \emph{peak} of height $i$ where $1 \leq i \leq n$ if there exists a $u$ directly followed by a $d$ in which the $u$ has ending height $i$. Let $\mathcal{E}_n$ denote the subset of $\mathcal{D}_n$ whose members contain no peaks of height $1$.  Then it is well-known that $t_n=\mathcal{E}_{n-1}$ for $n \geq 1$, with $t_0=0$.

By a \emph{return} within a member of  $\mathcal{P}_n$, we mean an $h$ or $u$ step that terminates on the $x$-axis.  A \emph{terminal} return is one that has endpoint $(2n,0)$, with all other returns being referred to as \emph{non-terminal}. By a \emph{unit} within $\lambda \in \mathcal{P}_n$, we mean a subpath of $\lambda$ occurring between two adjacent returns or prior to the first return.  Note that a low $h$ step comprises its own unit with all other units of the form $u\sigma d$ for some possibly empty Schr\"{o}der path $\sigma$.  Within members of $\mathcal{E}_n$, all units must have length at least two, whereas members of $\mathcal{Q}_n$ can also contain units of the form $ud$, but not $h$.  Finally, a member of $\mathcal{P}_n$ having no non-terminal returns is said to be \emph{primitive}. A primitive member $\lambda \in \mathcal{P}_n$ for $n \geq 2$ is necessarily of the form $\lambda=u\sigma d$, where $\sigma \in \mathcal{P}_{n-1}$, and hence belongs to $\mathcal{Q}_n$.

We compute the determinant of an $n\times n$ Hessenberg--Toeplitz matrix using the definition of a determinant as a signed sum over the set of permutations $\sigma$ of $[n]$.  In doing so, one need only consider those $\sigma$ whose cycles when expressed disjointly each comprise a set of consecutive integers.  Such $\sigma$ are clearly in one-to-one correspondence with the compositions of $n$, upon identifying the various cycle lengths with parts of a composition.  This implies that the determinant of a matrix $A_n$ of the form \eqref{DetHess} may be regarded as a weighted sum over the set of compositions of $n$.  If $a_0=1$ in this sum, then each part of size $r \geq 1$ has (signed) weight given by $(-1)^{r-1}a_r$ (regardless of its position) and the weight of a composition is the product of the weights of its constituent parts.  One can then define the sign of a composition as $(-1)^{n-m}$, where $m$ denotes the number of its parts.  On the other hand, when $a_0=-1$, every part of size $r$ now contributes $a_r$ towards the weight of the composition.  Thus, assuming $a_i \geq 0$ for $i \geq 1$, each term in the determinant sum for $A_n$ is non-negative in this case.  Note that computing $\det(A_n)$ where $a_0=-1$ is equivalent to finding the permanent of the matrix obtained from $A_n$ by replacing $a_0=-1$ with $a_0=1$.

We now provide combinatorial proofs of the formulas from Theorems \ref{Theo1} and \ref{Theo3} above. \medskip

\noindent\textbf{Proofs of \eqref{Theo1e1}, \eqref{Theo1e2}, \eqref{Theo1e5} and \eqref{Theo1e6}:}\medskip

Let $\mathcal{A}_n$ denote the set of marked Schr\"{o}der paths of length $n$ in which returns to the $x$-axis may be marked and whose final return is always marked.  Define the sign of $\lambda \in \mathcal{A}_n$ by $(-1)^{n-\mu(\lambda)}$, where $\mu(\lambda)$ denotes the number of marked returns of $\lambda$.  Let $\mathcal{A}_n'\subseteq\mathcal{A}_n$ consist of those members of $\mathcal{A}_n$ in which there are no low $h$ steps (marked or unmarked). Then $D_+(S_1,\ldots,S_n)$ and $D_+(s_1,\ldots,s_n)$ give the sum of the signs of all members of $\mathcal{A}_n$ and $\mathcal{A}_n'$, respectively.  To see this, first suppose $\tau$ is a member of $\mathcal{A}_n$ or $\mathcal{A}_n'$ and is derived from the (weighted) composition $\sigma$ in either determinant expansion.  That is, $\tau$ is obtained from $\sigma$ by overlaying a member of $\mathcal{P}_r$ or $\mathcal{Q}_r$ on each part of $\sigma$ of size $r$ for every $r$, marking the final return of each path and finally concatenating the paths in the same order as the parts of $\sigma$.  Then the sequence of part sizes of $\sigma$ corresponds to the sequence of lengths of the subpaths occurring between adjacent marked returns of $\tau$ (or prior to the first marked return), and, in particular, the number of parts of $\sigma$ equals the number of marked returns of $\tau$.  Thus, the sign of $\sigma$ in the determinant expansion corresponds to $n-\mu(\tau)$, and considering all $\tau$ associated with each $\sigma$ implies $D_+(S_1,\ldots,S_n)$ and $D_+(s_1,\ldots,s_n)$ give the sum of the signs of the members of $\mathcal{A}_n$ and $\mathcal{A}_n'$, respectively, as claimed.

We define a sign-changing involution on $\mathcal{A}_n$ by identifying the leftmost non-terminal return and either marking it or removing the marking from it.  The set of survivors of this involution consists of the \emph{primitive} members of $\mathcal{A}_n$.  If $n \geq 2$, then there are $S_{n-1}$ primitive members of $\mathcal{A}_n$, each of sign $(-1)^{n-1}$, which implies \eqref{Theo1e1}.  Since the survivors of the involution all belong to $\mathcal{A}_n'$, this establishes \eqref{Theo1e5} as well.

On the other hand, it is seen from the preceding that $D_-(S_1,\ldots,S_n)$ and $D_-(s_1,\ldots,s_n)$ give the cardinalities of the sets $\mathcal{A}_n$ and $\mathcal{A}_n'$, respectively, since when $a_0=-1$ the sign of $\sigma$ is cancelled out by the product of the superdiagonal $-1$ factors in the term corresponding to $\sigma$ in the determinant expansion.  We first show \eqref{Theo1e6}.  Let $\mathcal{P}_n^*$ denote the set of colored members of $\mathcal{P}_n$ wherein each low $h$ step is colored in one of three ways.  Recall one of the combinatorial interpretations of $A225887[n]$ is that it gives the cardinality of $\mathcal{P}_{n}^*$ for $n \geq 0$.  Thus, to complete the proof of \eqref{Theo1e6}, it suffices to define a bijection $\phi: \mathcal{P}_{n-1}^* \rightarrow \mathcal{A}_n'$.  Let $h_a$, $h_b$, $h_c$ denote the three kinds of colored low $h$ steps within $\lambda \in \mathcal{P}_{n-1}^*$.  We decompose $\lambda$ as $\lambda=\lambda^{(1)}\cdots \lambda^{(r)}$ for some $r \geq 1$, where each subpath $\lambda^{(i)}$ for $1 \leq i \leq r-1$ ends in either $h_a$ or $h_b$, with all other low $h$ steps in $\lambda$ (if there are any) equal to $h_c$, and $\lambda^{(r)}$ is possibly empty.  Note that if $\lambda$ contains no $h_a$ or $h_b$ steps, then we take $r=1$ and $\lambda=\lambda^{(1)}$; further, if $\lambda$ ends in $h_a$ or $h_b$, then $r \geq 2$ with $\lambda^{(r)}$ understood to be empty in this case. If $1 \leq i \leq r-1$ and $\lambda^{(i)}$ ends in $h_a$ with $\lambda^{(i)}=\alpha_ih_a$, where $\alpha_i$ is possibly empty, then let $\overline{\lambda}^{(i)}=u\alpha_id$, where the final $d$ is marked (i.e., the return to the $x$-axis associated with this $d$ is marked). If $\lambda^{(i)}=\beta_ih_b$, then let $\overline{\lambda}^{(i)}=u\beta_id$, where the final $d$ is unmarked.  Finally, let $\overline{\lambda}^{(r)}=u\lambda^{(r)}d$, where the final $d$ is marked.  Define $\phi(\lambda)=\overline{\lambda}^{(1)}\cdots\overline{\lambda}^{(r)}$ as the concatenation of the lattice paths $\overline{\lambda}^{(i)}$.  Note that $\phi$ can be reversed, and hence is bijective, as desired, upon considering the positions of the returns and whether or not they are marked.  Further, it is seen that the number of $h_c$ steps within $\lambda$ equals the number of $h$ steps of height $1$ within $\phi(\lambda)$ for all $\lambda$.

We now show \eqref{Theo1e2}. Let $\mathcal{\widetilde{P}}_n$ denote the set derived from members of $\mathcal{P}_n$ by stipulating that the low $h$ steps come in one of four kinds, denoted by $h^{(i)}$ for $1 \leq i \leq 4$.  Recall that $A134425[n]$ gives $|\mathcal{\widetilde{P}}_n|$ for $n \geq0$, so for \eqref{Theo1e2}, we need to prove $|\mathcal{A}_n|=2|\mathcal{\widetilde{P}}_{n-1}|$ for $n \geq 1$. We proceed inductively, noting that the $n=1$ case of the equality is clear.  Let $n \geq 2$ and we consider the following cases on members $\lambda \in \mathcal{A}_n$: (i) $\lambda=\lambda'h$, (ii) $\lambda=\lambda'\alpha$, where $\alpha\neq h$ is a unit and $\lambda'$ is nonempty, with the final return of $\lambda'$ marked, or (iii) $\lambda=\lambda'\beta$, where $\beta \neq h$ is a unit and either $\lambda'=\varnothing$ or $\lambda'\neq \varnothing$ with the final return of $\lambda'$ not marked.  We partition $\rho \in \mathcal{\widetilde{P}}_{n-1}$ as follows: (I) $\rho$ ends in $h^{(1)}$ or $h^{(2)}$, (II)  $\rho=\rho'\alpha$, where $\alpha\neq h^{(i)}$ for any $i$ is a unit and $\rho'$ is possibly empty, or (III) $\rho$ ends in $h^{(3)}$ or $h^{(4)}$.

We now demonstrate for each of (i)--(iii) that there are twice as many members $\lambda \in \mathcal{A}_n$ as there are  $\rho \in \mathcal{\widetilde{P}}_{n-1}$ in the corresponding case (I)--(III). Upon considering whether or not the final return in $\lambda'$ is marked, it is seen by the induction hypothesis that there are twice as many $\lambda \in \mathcal{A}_n$ for which (i) applies as there are $\rho \in \mathcal{\widetilde{P}}_{n-1}$ for which (I) applies.  The same holds true of (ii) and (II) as $\lambda'$ in (ii) has length one greater than that of $\rho'$ in (II), with $\alpha$ the same in both cases.  To show that the same holds for cases (iii) and (III) above, observe first that the number of possible $\rho \in \mathcal{\widetilde{P}}_{n-1}$ in (III) is given by $2|\mathcal{\widetilde{P}}_{n-2}|$. Thus, to complete the proof of \eqref{Theo1e2}, it is enough to prove that there are $2|\mathcal{A}_{n-1}|$ possible $\lambda \in \mathcal{A}_n$ in (iii).

Let $\lambda=\lambda'\beta \in \mathcal{A}_n$, where $\beta \neq h$ is a unit and $\lambda'$ does not have a marked final return.  If $\lambda'=\varnothing$, i.e., $\lambda$ is primitive, then write $\beta=u\beta'd$ and regard $\beta'$ as a member of $\mathcal{A}_{n-1}$ in which only the final return is marked.  Otherwise, consider cases based on the length $\ell$ of $\beta$, where $1 \leq \ell \leq n-1$. If $\ell=1$, i.e., $\beta=ud$, then regard $\lambda'$ as a member of $\mathcal{A}_{n-1}$ by marking its last return.  If $\ell \geq 2$, then let $\beta=u\beta'd$, where $\beta'$ is nonempty.  Then form the lattice path $\sigma=\lambda'\beta'$ of length $n-1$, wherein the last return of $\lambda'$ and of $\beta'$ are now both marked (here, it is understood that all other returns of $\lambda'$ remain of the same status regarding whether or not they are marked and that all non-terminal returns of $\beta'$, if any, are unmarked).  Note that $\sigma \in \mathcal{A}_{n-1}$ with $\sigma$ containing at least two marked returns.  It is seen then that each member of $\mathcal{A}_{n-1}$ arises exactly twice when one performs the operations described above on the various members of $\mathcal{A}_n$ for which (iii) applies, upon considering whether or not a member of $\mathcal{A}_{n-1}$ contains two or more marked returns, and if it does, additionally taking into account the position of the rightmost non-terminal marked return.  This establishes the desired equality
$|\mathcal{A}_n|=2|\mathcal{\widetilde{P}}_{n-1}|$ for all $n \geq 1$, which completes the proof of \eqref{Theo1e2}.
\hfill \qed \medskip

\noindent\textbf{Proofs of \eqref{Theo1e3}, \eqref{Theo1e4}, \eqref{Theo1e7} and \eqref{Theo1e8}:}\medskip

Let $\mathcal{A}_n$ be as in the previous proof and let $\mathcal{B}_n\subseteq \mathcal{A}_n$ consist of those members in which all marked returns (including the final return) correspond to low $h$ steps. Let $\mathcal{B}_n'\subseteq \mathcal{B}_n$ consist of those members in which no low $h$ is unmarked.  Define the sign of $\lambda \in \mathcal{B}_n$ by $(-1)^{n-\mu(\lambda)}$, where $\mu(\lambda)$ denotes the number of marked low $h$'s.  Reasoning as in the prior proof, we have that $D_+(S_0,\ldots,S_{n-1})$ and $D_+(s_0,\ldots,s_{n-1})$ give the sum of the signs of all members of $\mathcal{B}_n$ and $\mathcal{B}_n'$, respectively.  To show  \eqref{Theo1e3}, we define a sign-changing involution on $\mathcal{B}_n$ by identifying the leftmost non-terminal low $h$ and either marking it or removing the marking from it.  This involution fails to be defined for paths of the form $\lambda=\alpha h$, where $\alpha \in \mathcal{Q}_{n-1}$ and $h$ is marked.  Thus, there are $s_{n-1}$ survivors of the involution, each of sign $(-1)^{n-1}$, which implies \eqref{Theo1e3}.  For \eqref{Theo1e7}, we define an involution on $\mathcal{B}_{n}'$ by identifying the leftmost non-terminal (marked) low h step or peak of height $1$ (i.e., unit of the form $ud$) and replacing one option with the other.  This involution is not defined on members $\rho=\beta h$, where $\beta\in\mathcal{P}_{n-1}$ contains no low $h$ steps or peaks of height $1$.  Note that there are $A114710[n-1]$ such $\rho$ for all $n \geq 1$, each with sign $(-1)^{n-1}$, which implies \eqref{Theo1e7}.

On the other hand, we have that $D_-(S_0,\ldots,S_{n-1})$ and $D_-(s_0,\ldots,s_{n-1})$ give the cardinalities of the sets $\mathcal{B}_n$ and $\mathcal{B}_n'$, respectively.  To show \eqref{Theo1e4}, consider decomposing $\rho \in \mathcal{B}_n$ as $\rho=\rho^{(1)}\cdots\rho^{(r)}$ for some $r \geq 1$, where each $\rho^{(i)}$ ends in a marked low $h$ step and contains no other marked steps. Write $\rho^{(i)}=\alpha_ih$ for $1 \leq i \leq r$, where $\alpha_i$ is possibly empty. Define $\overline{\rho}^{(i)}=u\alpha_i d$ and let $\overline{\rho}=\overline{\rho}^{(1)}\cdots\overline{\rho}^{(r)}$. Then the mapping $\rho \mapsto \overline{\rho}$ is seen to define a bijection between $\mathcal{B}_n$ and $\mathcal{Q}_n$ (to reverse it, consider positions of the returns in members of $\mathcal{Q}_n$), and hence $|\mathcal{B}_n|=s_n$, which implies \eqref{Theo1e4}.  Finally, members of $\mathcal{B}_n'$ and $\mathcal{P}_{n-1}$ are seen to be synonymous, upon removing the marking from all low $h$'s and disregarding the final $h$ in members of the former, which implies \eqref{Theo1e8}. \hfill \qed \medskip

\noindent\textbf{Proof of \eqref{Theo1e9}:}\medskip

Let $\mathcal{J}_{n,k}$ for $1 \leq k \leq n$ denote the set of ordered $k$-tuples $\lambda=(\lambda_1,\ldots,\lambda_k)$ wherein each $\lambda_i$ is a restricted Schr\"{o}der path having length at least two such that $\sum_{i=1}^k|\lambda_i|=n+k$.  Define the sign of $\lambda \in \mathcal{J}_{n,k}$ by $(-1)^{n-k}$ and let $\mathcal{J}_n=\cup_{k=1}^n\mathcal{J}_{n,k}$.  Then we have that $D_+(s_2,\ldots,s_{n+1})$ gives the sum of the signs of all members of $\mathcal{J}_n$.  We define a sign-changing involution of $\mathcal{J}_n$ which makes use of several cases as follows.  First suppose that the final component $\lambda_k$ of $\lambda \in \mathcal{J}_{n,k}$ is \emph{not} primitive. If $\lambda_k=u\sigma d\tau$, where $\sigma$ is a possibly empty Schr\"{o}der path and $|\tau|\geq 2$, then replace $\lambda_k$ with the two components $\lambda_k=\tau$ and $\lambda_{k+1}=u\sigma dud$, leaving all other components of $\lambda$ unchanged.  We perform the reverse operation, i.e., fusing the last two components and dropping $ud$, if the last component consists of a unit followed by $ud$.  This pairs all members of $\mathcal{J}_n$ in which the final component is not primitive except for those belonging to $\mathcal{J}_{n,1}$ where $\lambda_1=u\sigma dud$ for some $\sigma$.

Now suppose $\lambda_k$ within $\lambda$ is primitive. First assume $|\lambda_k|\geq 3$, and we consider the following further subcases:
\begin{align*}
(\text{i})~&\lambda_k=u\sigma d, \text{ with } \sigma \text{ containing no low } h\text{'s} \text{ and } |\sigma|\geq 2,\\
(\text{ii})~&\lambda_k=u\sigma'h\sigma''d, \text{ with } \sigma'\neq \varnothing \text{ and containing no low } h\text{'s} \text{ and } \sigma'' \text{ possibly empty},\\
(\text{iii})~&\lambda_k=uh\sigma d, \text{ with } \sigma\neq\varnothing,
\end{align*}
where $\sigma,\sigma',\sigma''$ denote Schr\"{o}der paths.  (Note that by $\sigma$ or $\sigma'$ not containing a low $h$ in the preceding, we mean when $\sigma$ or $\sigma'$ is viewed by itself starting from the origin.)  Now suppose $\rho=(\rho_1,\ldots,\rho_k) \in \mathcal{J}_{n,k}$, with $\rho_k$ primitive and $|\rho_k|=2$.  We consider the following subcases: (I) $\rho_k=u^2d^2$, (II) $\rho_k=uhd$, with $\rho_{k-1}$ not primitive, or (III) $\rho_k=uhd$, with $\rho_{k-1}$ primitive.  Note that $n \geq 2$ implies $k \geq 2$ in (I)--(III) and hence a penultimate component exists in each case. We now perform the following operations on the members of $\mathcal{J}_{n,k}$ in (i)--(iii) above (leaving all other components unchanged):
\begin{align*}
(\text{a})~&\lambda_k=u\sigma d \leftrightarrow \lambda_k=\sigma,~\lambda_{k+1}=u^2d^2,\\
(\text{b})~&\lambda_k=u\sigma'h\sigma''d  \leftrightarrow \lambda_k=\sigma'u\sigma''d,~\lambda_{k+1}=uhd,\\
(\text{c})~&\lambda_k=uh\sigma d \leftrightarrow \lambda_k=u\sigma d,~\lambda_{k+1}=uhd.
\end{align*}
Note that the assumptions on $\sigma,\sigma',\sigma''$ in (i)--(iii) imply that these operations are well-defined and it is seen that they are reversible in each case.  Hence, they provide bijections between the members of $\mathcal{J}_n$ satisfying (i), (ii) or (iii) and those satisfying (I), (II) or (III), respectively.  Since the number of components changes by one in all cases, each member of $\mathcal{J}_n$ whose final component is primitive is paired with another of opposite sign.  Thus, when taken together with the pairing defined in the preceding paragraph, we have that all members of $\mathcal{J}_n$ are paired except for $\lambda=(\lambda_1) \in \mathcal{J}_{n,1}$ such that $\lambda_1=u\sigma dud$ for some $\sigma \in \mathcal{P}_{n-1}$.  There are $S_{n-1}$ possibilities for these $\lambda$, each having sign $(-1)^{n-1}$, which implies formula \eqref{Theo1e9}.  \hfill \qed \medskip

\noindent\textbf{Proofs of \eqref{Theo3e1} and \eqref{Theo3e2}:}\medskip

We first find a combinatorial interpretation for $D_-(t_1,\ldots,t_n)$.  A \emph{short} unit within a member of $\mathcal{D}_n$ will refer to a unit having length one (i.e., is equal $ud$), with all other units being referred to as \emph{long}. Let $\mathcal{D}_n'$ denote the subset of $\mathcal{D}_n$ whose members have last unit short and hence $|\mathcal{D}_n'|=C_{n-1}$ for $n \geq 1$.  Suppose $\rho$ is a (weighted) composition of $n$ with $m$ parts occurring in the expansion of $D_-(t_1,\ldots,t_n)$.  On a part of size $r$ within $\rho$, we overlay $\alpha \in \mathcal{E}_{r-1}$ followed by $ud$.  We do this for each part of $\rho$ and concatenate the resulting lattice paths $\alpha ud$ to obtain a member of $\mathcal{D}_n'$ in which there are $m$ short units altogether.  Upon considering all possible $m$, we have that $D_{-}(t_1,\ldots,t_n)$ gives the cardinality of $\mathcal{D}_n'$, which implies \eqref{Theo3e2}.

To show \eqref{Theo3e1}, first note that $D_+(t_1,\ldots,t_n)$ gives the sum of the signs of all $\lambda \in \mathcal{D}_{n}'$, where the sign of $\lambda$ is defined as $(-1)^{n-\nu(\lambda)}$ and $\nu$ denotes the statistic recording the number of short units.  Let $r_n=D_+(t_1,\ldots,t_n)$ for $n \geq 1$; clearly, we have $r_1=r_2=1$, so we may assume $n \geq 3$.  Let $\rho \in \mathcal{D}_n'$.  If the first unit of $\rho$ has length $i+1$ for some $1 \leq i \leq n-2$, then the contribution towards the sum of signs is given by $(-1)^{i+1}C_ir_{n-i-1}$.  Summing over all $i$ yields a total contribution of $\sum_{i=1}^{n-2}(-1)^{i+1}C_ir_{n-i-1}$ for members of $\mathcal{D}_n'$ whose first unit is long.  On the other hand, if the first unit is short, then there are $r_{n-1}$ possibilities as no adjustment for the sign is required when prepending a short unit to a member of $\mathcal{D}_{n-1}'$.  Combining the prior cases of $\rho$ implies $r_n$ satisfies the desired recurrence and completes the proof.  \hfill \qed \medskip

\noindent\textbf{Proofs of \eqref{Theo3e3} and \eqref{Theo3e4}:}\medskip

Let $\mathcal{L}_n$ denote the set of marked members of $\mathcal{E}_n$ wherein the first unit is not marked and all other units may be marked.  Define the sign of $\lambda \in \mathcal{E}_n$ by $(-1)^{n-\mu(\lambda)}$, where $\mu(\lambda)$ denotes the number of unmarked units of $\lambda$. Then $D_+(t_2,\ldots,t_{n+1})$ and $D_-(t_2,\ldots,t_{n+1})$ are seen to give the sum of signs and cardinality, respectively, of the members of $\mathcal{L}_n$. To show \eqref{Theo3e3}, define an involution on $\mathcal{L}_n$ by marking or unmarking the second unit, if it exists.  This operation is not defined on the primitive members of $\mathcal{L}_n$, each of which has sign $(-1)^{n-1}$.  Since the primitive members of $\mathcal{L}_n$ have cardinality $C_{n-1}$ for $n \geq 2$, formula \eqref{Theo3e3} is established.

To show \eqref{Theo3e4}, let $b_n=D_-(t_2,\ldots,t_{n+1})$ for $n \geq 1$ and note
\begin{equation}\label{bneq1}
b_n=C_{n-1}+2\sum_{k=1}^{n-3}C_kb_{n-k-1}, \qquad n \geq 3,
\end{equation}
with $b_1=0$ and $b_2=1$, upon considering whether or not a member of $\mathcal{L}_n$ is primitive and, if not, taking into account the length $k+1$ of the first unit, where $1 \leq k \leq n-3$.  Here, the factor of 2 accounts for the choice concerning whether or not the second unit is marked in the latter case.  In order to establish $b_n=A137398[n]$, we must show that $b_n$ satisfies the defining recurrence for $A137398[n]$, i.e.,
\begin{equation}\label{bneq2}
b_n=2b_{n-1}+2b_{n-2}+\sum_{k=1}^{n-3}C_kb_{n-k-1}, \qquad n \geq 4.
\end{equation}
Comparing \eqref{bneq1} and \eqref{bneq2}, to complete the proof of \eqref{Theo3e4}, it suffices to show
\begin{equation}\label{bneq3}
C_{n-1}+\sum_{k=2}^{n-3}C_kb_{n-k-1}=2b_{n-1}+b_{n-2}, \qquad n \geq 4.
\end{equation}
We may assume $n \geq 5$ in \eqref{bneq3} since it is seen to hold for $n=4$.

To prove \eqref{bneq3}, we describe a combinatorial structure enumerated by the left side of the identity and show that this structure is also enumerated by the right. We will make use of the same descriptors short and long as before when referring to units of varying length. Let $\mathcal{Y}_n$ denote the set of all marked Dyck paths of length $n$ containing at least one short unit wherein long units occurring to the right of the rightmost short unit (if there are any) may be marked, but where the first such long unit is always unmarked.  Further, we require that the rightmost short unit within a member of $\mathcal{Y}_n$ correspond to the $(2i-1)$-st and $(2i)$-th steps for some $i \geq 3$.  Note that there are $C_{n-1}$ members of $\mathcal{Y}_n$ ending in a short unit, upon appending $ud$ to any member of $\mathcal{D}_{n-1}$.  Otherwise, $\lambda \in \mathcal{Y}_n$ is expressible as $\lambda=\lambda'ud\lambda''$, where $\lambda'$ is any Dyck path with $|\lambda'|\geq 2$ and $\lambda''$ is nonempty and consists of long units that may be marked, except for the first, which is always unmarked.  Then there are $C_kb_{n-k-1}$ possibilities for $\lambda$ in which $|\lambda'|=k$ and considering all possible $k \in [2,n-3]$ implies that there are $\sum_{k=2}^{n-3}C_kb_{n-k-1}$ members of $\mathcal{Y}_n$ that end in a long unit.  Thus, we have that the left-hand side of \eqref{bneq3} gives $|\mathcal{Y}_n|$.

We now show that $2b_{n-1}+b_{n-2}$ also gives $|\mathcal{Y}_n|$.  First let us take two copies of each $\alpha \in \mathcal{L}_{n-1}$, where it is assumed for now that $\alpha$ contains at least one marked unit.  Then write $\alpha=\alpha_1\cdots\alpha_{\ell-1}\alpha_\ell\cdots\alpha_r$, where the $\alpha_i$ denote the units of $\alpha$, the leftmost marked unit is $\alpha_\ell$ and $2 \leq \ell \leq r$.  Within the first copy of $\alpha$, we insert $ud$ directly between the units $\alpha_{\ell-1}$ and $\alpha_\ell$.  Within the second copy of $\alpha$, we replace $\alpha_{\ell-1}$ with $ud\alpha_{\ell}'ud$, where $\alpha_{\ell-1}=u\alpha_{\ell-1}'d$.  In both cases, we remove the mark from the unit $\alpha_\ell$ and leave all other units of $\alpha$ undisturbed. On the other hand, if $\alpha \in \mathcal{L}_{n-2}$ contains a marked unit and is decomposed into units as above, then we insert $udud$ between the units $\alpha_{\ell-1}$ and $\alpha_\ell$ and remove the mark from $\alpha_\ell$.  Note that the operations described in this paragraph yield uniquely all members of $\mathcal{Y}_n$ not ending in $ud$ and can be reversed by considering the position of the rightmost short unit and taking into account whether there are one or more short units.  If there are more than one,  then consider further whether or not the leftmost and rightmost short units are adjacent.

So it remains to show
$$2|\{\alpha \in \mathcal{L}_{n-1}: \alpha \text{ has no marked units}\}|+|\{\alpha \in \mathcal{L}_{n-2}: \alpha \text{ has no marked units}\}|$$
equals the number of members of $\mathcal{Y}_n$ ending in $ud$ (recall that this number is $C_{n-1}$).  Note that this equality is equivalent to the known relation $2t_n+t_{n-1}=C_{n-1}$ for $n \geq 2$; for a combinatorial proof, we refer the reader to \cite[Section~3]{DSh}.  This completes the proof of \eqref{bneq3}, as desired.\hfill \qed \medskip

\noindent\textbf{Proof of \eqref{Theo3e5}:}\medskip

Let $\mathcal{M}_{n,k}$ for $1 \leq k \leq n$ denote the set of ordered $k$-tuples $(\lambda_1,\ldots,\lambda_k)$ such that each $\lambda_i$ is a nonempty Dyck path all of whose units are long, with $\sum_{i=1}^k|\lambda_i|=n+k$.  Let members of $\mathcal{M}_{n,k}$ have sign $(-1)^{n-k}$.  Then it is seen that $D_+(t_3,\ldots,t_{n+2})$ gives the sum of the signs of all members of $\mathcal{M}_n$, where $\mathcal{M}_n=\cup_{k=1}^n\mathcal{M}_{n,k}$.  Before defining an involution on $\mathcal{M}_n$, let us recall a definition.  By a \emph{valley} of height $j$ within a Dyck path where $j \geq 0$, we mean a $d$ directly followed by a $u$ step in which the $u$ has starting height $j$.  A \emph{special} valley will refer to one of height $1$.  Let $\lambda=(\lambda_1,\ldots,\lambda_k) \in \mathcal{M}_{n,k}$ and suppose first that the component $\lambda_k$ contains at least one special valley.  We decompose $\lambda_k$ as $\lambda_k=\alpha\bf{du}\beta$, where $\alpha$ and $\beta$ contain $2a$ and $2b$ steps respectively and $\bf{du}$ denotes the rightmost special valley.  Note that $a,b \geq 1$, with $|\lambda_k|=a+b+1$.  Let $\lambda^*$ be obtained from $\lambda$ by replacing $\lambda_k$ with the two components $\lambda_k=\alpha d^2$ and $\lambda_{k+1}=u^2\beta$, keeping all other components of $\lambda$ the same.  One may verify $\lambda_k \in \mathcal{E}_{a+1}$, $\lambda_{k+1}\in\mathcal{E}_{b+1}$, and hence $\lambda^* \in \mathcal{M}_{n,k+1}$, with $\lambda_{k+1}$ containing no special valleys.  If it is the case that $\lambda \in M_{n,k}$ for some $k>1$ with $\lambda_k$ containing no special valleys, then $\lambda^*$ is obtained from $\lambda$ by reversing the operation described above.  The mapping $\lambda \mapsto \lambda^*$ is an involution of $\mathcal{M}_n$ which always changes the sign and is not defined on $\mathcal{M}_n' \subseteq \mathcal{M}_n$ consisting of those $\lambda=(\rho) \in \mathcal{M}_{n,1}$ such that $\rho$ contains no special valleys.

To enumerate the members of $\mathcal{M}_n'$, note that $\rho$ can be decomposed into units as $\rho=\rho_1\cdots\rho_j$ for some $j \geq 1$, where $\rho_i=u^2\rho_i'd^2$ for each $i$ with $\rho_i'$ possibly empty.  Let $a(n,j)$ denote the number of members of $\mathcal{D}_n$ that have $j$ returns.  Then removal of the initial $u$ and the final $d$ from each unit $\rho_i$ within $\rho$ implies that there are $a(n+1-j,j)$ possible $\rho$, and summing over all $j$ yields $|\mathcal{M}_n'|=\sum_{j=1}^{\lfloor(n+1)/2\rfloor}a(n+1-j,j)$.  Recall that one of the combinatorial properties for $A030238[n]$ is that it is given explicitly as $\sum_{j=1}^{\lfloor(n+2)/2\rfloor}a(n+2-j,j)$.  Hence, $|\mathcal{M}_n'|=A030238[n-1]$ for $n \geq 1$.  Since each member of $\mathcal{M}_n'$ has sign $(-1)^{n-1}$, the proof of \eqref{Theo3e5} is complete. \hfill \qed \medskip

\noindent\textbf{Proof of \eqref{Theo3e6}:}\medskip

Let $\mathcal{T}_{n,k}$ denote the set of ordered $k$-tuples $(\lambda_1,\ldots,\lambda_k)$ such that each $\lambda_i$ is a Dyck path of length at least three all of whose units are long, with $\sum_{i=1}^k|\lambda_i|=n+2k$.  Let members of $\mathcal{T}_{n,k}$ have sign $(-1)^{n-k}$ and let $\mathcal{T}_n=\cup_{k=1}^n\mathcal{T}_{n,k}$.  Then we have that $D_+(t_4,\ldots,t_{n+3})$ gives the sum of signs of all members of $\mathcal{T}_n$.  Let $\mathcal{T}_n'\subseteq \mathcal{T}_n$ consist of $(\lambda_1)\in \mathcal{T}_{n,1}$ such that $\lambda_1$ is expressible as $\lambda_1=u^2d^2\alpha$, where $\alpha$ is a unit.  Note that $n \geq 3$ implies $|\alpha| \geq 3$ and hence $\alpha$ is long, as required.  As there are $C_{n-1}$ possibilities for $\lambda_1$, we have $\sigma(\mathcal{T}_n')=(-1)^{n-1}C_{n-1}$, where $\sigma(S)$ denotes the sum of the signs of the members of a subset $S$ of $\mathcal{T}_n$.  Below, we define in several steps a sign-changing involution on the entirety of $\mathcal{T}_n-\mathcal{T}_n'$ when $n \geq 3$, which implies \eqref{Theo3e6}.

We first partition $\mathcal{T}_n-\mathcal{T}_n'$ into three subsets $\mathcal{U}_n$, $\mathcal{V}_n$ and $\mathcal{W}_n$ given by
\begin{align*}
(\text{i})~&\mathcal{U}_n=\{(\lambda_1,\ldots,\lambda_k) \in \mathcal{T}_n-\mathcal{T}_n': \lambda_k \ \text{ not primitive}\},\\
(\text{ii})~&\mathcal{V}_n=\{(\lambda_1,\ldots,\lambda_k) \in \mathcal{T}_n-\mathcal{T}_n': \lambda_k \ \text{ primitive and contains no special peaks}\},\\
(\text{iii})~&\mathcal{W}_n=\{(\lambda_1,\ldots,\lambda_k) \in \mathcal{T}_n-\mathcal{T}_n': \lambda_k \ \text{ primitive and contains at least one special peak}\},
\end{align*}
where $k \geq 1$ in each case and a \emph{special} peak is one of height two. We first define involutions on $\mathcal{U}_n$ and $\mathcal{V}_n$.  Let $(\lambda_1,\ldots,\lambda_k) \in \mathcal{U}_n$ and suppose $\lambda=\alpha\beta$, where $|\alpha|\geq 2$ and $\beta$ is a unit.  Then we replace the component $\lambda_k$ with the two components $\lambda_k=\alpha$ and $\lambda_{k+1}=u^2d^2\beta$, if $|\alpha| \geq 3$, or perform the inverse operation if $|\alpha|=2$ (i.e., $\alpha=u^2d^2$).  Note that the possible case where $k=1$ and $\lambda_1=u^2d^2\beta$ has been excluded from consideration since such members of $\mathcal{T}_n$ belong to $\mathcal{T}_n'$.  Thus, the two operations defined above taken together yield an involution, which we will denote by $\phi$, that is defined on all of $\mathcal{U}_n$.

Now suppose $(\lambda_1,\ldots,\lambda_k) \in \mathcal{V}_n$.  Then either $|\lambda_k| \geq 4$ and is primitive with no special peaks or $\lambda_k=u^3d^3$.  In the former case, we decompose $\lambda_k$ as  $\lambda_k=u\alpha d$, where $\alpha \geq 3$. If $|\lambda_k| \geq 4$, then replace the component $\lambda_k=u\alpha d$ with the two components $\lambda_k=\alpha$ and $\lambda_{k+1}=u^3d^3$, keeping all other components the same.  Note that $\lambda_k$ containing no special peaks implies that the penultimate component $\alpha$ in the resulting member of $\mathcal{T}_n$ contains no short units, as required.  If the final component $\lambda_k$ equals $u^3d^3$, then perform the inverse operation, noting that $n \geq 3$ implies $k \geq 2$ in this case.  Thus, the two operations taken together yield an involution, which we will denote by $\psi$, that is defined on all of $\mathcal{V}_n$.

Define the subset $\mathcal{W}_n(1)$ of $\mathcal{W}_n$ as follows:
\begin{align*}
\mathcal{W}_n(1)=\{(\lambda_1,\ldots,\lambda_k) \in \mathcal{W}_n:&\,\lambda_k=u\alpha ud\beta d, \text{ where } |\alpha|\geq 1 \text{ and } \beta \text{ contains only long units }\\
& \text{ and is possibly empty}\}.
\end{align*}
In Lemma \ref{Lem1} below, it is shown $\sigma(W_n(1))=0$.

Now define the subset $\mathcal{W}_n(2)$ of $\mathcal{W}_n$ as consisting of those $(\lambda_1,\ldots,\lambda_k)$ such that one of the following two conditions holds:
\begin{align*}
(\text{a})~&k\geq1 \text{ and } \lambda_k=u(ud)\beta d, \text{ where } \beta \text{ consists of two or more long units, or}\\
(\text{b})~&k\geq2 \text{ and } \lambda_k=u(ud)\tau d, \text{ where } \tau \text{ is a single long unit, and } \lambda_{k-1}=u(ud)\beta d, \text{ where } \beta \text{ consists }\\
&\text{of one or more long units}.
\end{align*}
Define an involution of $\mathcal{W}_n(2)$ by breaking apart or combining the final two components as indicated:
$$\lambda_k=u(ud)\beta d \leftrightarrow \lambda_k=u(ud)\beta' d,\,\lambda_{k+1}=u(ud)\tau d,$$
where $\beta$ consists of two or more long units, the first of which is denoted by $\tau$, and $\beta'=\beta-\tau$.

Let $\mathcal{W}_n'=\mathcal{W}_n-\mathcal{W}_n(1)-\mathcal{W}_n(2)$.  Note that $(\lambda_1,\ldots,\lambda_k) \in \mathcal{W}_n'$ implies $\lambda_k=u(ud)\tau d$, where $\tau$ is a long unit. We decompose $\mathcal{W}_n'$ as $\mathcal{W}_n'=\cup_{i=1}^4\mathcal{W}_n'(i)$, where $\mathcal{W}_n'(i)$ for $1 \leq i \leq 4$ consists of those $(\lambda_1,\ldots,\lambda_k)$ in $\mathcal{W}_n'$ satisfying respectively
\begin{align*}
(\text{1})~&k=1,\\
(\text{2})~&k\geq2 \text{ and } \lambda_{k-1} \text{ is not primitive},\\
(\text{3})~&k\geq2 \text{ and } \lambda_{k-1} \text{ is primitive with no special peaks, or}\\
(\text{4})~&k \geq 2 \text{ and } \lambda_{k-1}=u\alpha(ud)\beta d, \text{ where } |\alpha|\geq 1 \text{ and } \beta, \text{ possibly empty, consists of long units}.
\end{align*}

Below, it is shown in Lemma \ref{Lem2} that $\sigma(\mathcal{W}_n')=0$ using the cases above, and hence $\sigma(\mathcal{W}_n)=0$. This implies $\sigma(\mathcal{T}_n-\mathcal{T}_n')=0$, as desired.
\hfill \qed \medskip

\begin{lemma}\label{Lem1}
If $n \geq 2$, then $\sigma(\mathcal{W}_n(1))=0$.
\end{lemma}
\begin{proof}
The result is readily shown if $n =2$, so we may assume $n \geq 3$. We pair members of $\mathcal{W}_n(1)$ of opposite sign by either breaking apart the last component or combining the last two components as indicated:
$$\lambda_k=u\alpha ud\beta d, \, |\alpha|\geq 2 \leftrightarrow \lambda_k=u\alpha d,\,\lambda_{k+1}=u(ud)^2\beta d.$$
The set of survivors of this involution consists of those $k$-tuples $(\lambda_1,\ldots,\lambda_k)$ such that either (i) $k \geq 2$ and $\lambda_k=u(ud)^2\beta d$, with $\beta$ consisting of long units if nonempty and $\lambda_{k-1}$ not primitive, or (ii) $k=1$ and $\lambda_1=u(ud)^2\beta d$, with $\beta$ as in (i).  Note that $n \geq 3$ implies $\beta \neq \varnothing$ in the latter case.  On the survivors satisfying condition (i), we apply the involution $\phi$ defined above to the $(k-1)$-tuple comprising the first $k-1$ components and then append $\lambda_k$ to the resulting vector.  Thus, all members satisfying (i) are paired except for those in which $k=2$ with $\lambda_1=u^2d^2\tau$ and $\lambda_2=u(ud)^2\beta d$, where $\beta$ consists of long units and $\tau$ is a single (long) unit.

Suppose $|\tau|=i+1$ in the decomposition of $\lambda_1$. This implies $$|\beta|=(n+4)-|\lambda_1|-3=n-2-i$$ in $\lambda_2$, and thus $\beta \in \mathcal{E}_{n-2-i}$. Hence summing over all $i$ yields $\sum_{i=1}^{n-2}C_it_{n-1-i}$  possible ordered pairs $(\lambda_1,\lambda_2)$. Further, the survivors in case (ii) above have cardinality $t_n$ since $\beta$ has length $n-1$ and contains only long units.  Thus, the sum of the signs of the remaining unpaired members of $\mathcal{W}_n(1)$ is given by
$$(-1)^{n-2}\sum_{i=1}^{n-2}C_it_{n-1-i}+(-1)^{n-1}t_n=0,$$
as desired, upon observing the recurrence $t_n=\sum_{i=1}^{n-2}C_it_{n-1-i}$ for $n \geq 3$.  Note that this recurrence may be easily realized combinatorially by considering the length $i+1$ of the first unit within a member of $\mathcal{E}_{n-1}$.  Thus, if desired, it is straightforward to pair the remaining members of $\mathcal{W}_n(1)$ of opposite sign upon considering the position of the first return within a member of $\mathcal{E}_{n-1}$.
\end{proof}

\begin{lemma}\label{Lem2}
If $n \geq 3$, then $\sigma(\cup_{i=1}^3\mathcal{W}_n'(i))=-\sigma(\mathcal{W}_n'(4))=(-1)^{n-1}C_{n-2}$, and hence $\sigma(\mathcal{W}_n')=0$.
\end{lemma}
\begin{proof}
We consider several cases on $\lambda=(\lambda_1,\ldots,\lambda_k) \in \mathcal{W}_n'$ whose last component $\lambda_k$ is given by $\lambda_k=u(ud)\tau d$, where $\tau$ is a long unit. If $\lambda \in \mathcal{W}_n'(1)$, then $k=1$ implies $|\tau|=n$ and thus $\sigma(\mathcal{W}_n'(1))=(-1)^{n-1}C_{n-1}$.  If $\lambda \in \mathcal{W}_n'(2)$, we apply the mapping $\phi$ defined above to $\lambda'=(\lambda_1,\ldots,\lambda_{k-1})$ and then append $\lambda_k$ to $\phi(\lambda')$.  This operation yields an involution on $\mathcal{W}_n'(2)$ that is not defined for those members in which $k=2$ with $\lambda_1=u^2d^2\sigma$ and $\sigma$ is a unit.  Upon considering $|\sigma|=i+1$ for $1 \leq i \leq n-3$, one gets
$$\sum_{i=1}^{n-3}C_iC_{n-2-i}=C_{n-1}-2C_{n-2}$$
unpaired members of $\mathcal{W}_n'(2)$, by the recurrence for the Catalan numbers.  If $\lambda \in \mathcal{W}_n'(3)$, we apply the mapping $\psi$ defined above to $\lambda'$ and then append $\lambda_k$ to $\psi(\lambda')$.  This operation yields an involution on  $\mathcal{W}_n'(3)$ except for those members where $k=2$ and $\lambda_1=u^3d^3$, of which there are $C_{n-2}$ possibilities.  Combining the contributions from $W_n'(i)$ for $1 \leq i \leq 3$ yields
$$\sigma(\cup_{i=1}^3W_n'(i))=(-1)^{n-1}C_{n-1}+(-1)^{n-2}(C_{n-1}-2C_{n-2})+(-1)^{n-2}C_{n-2}=(-1)^{n-1}C_{n-2}.$$

For the second statement, let $T$ denote the subset of $\mathcal{W}_n'(4)$ consisting of those members where $k=2$ and $\lambda_1=u(ud)^2d$.  Since $\sigma(T)=(-1)^{n-2}C_{n-2}$, we need to show $\sigma(\mathcal{W}_n'(4)-T)=0$.  Note that within the final component $\lambda_k=u(ud)\tau d$ of $\lambda \in \mathcal{W}_n'(4)-T$, we must have $2 \leq |\tau|\leq n-2$.  We may then apply the involution $g$ from Lemma \ref{Lem1} to $\lambda'$ (as $|\tau|\leq n-2$), and to the resulting vector $g(\lambda')$, we append the component $\lambda_k$.  This operation is seen to yield a sign-changing involution of $\mathcal{W}_n'(4)-T$, which completes the proof.
\end{proof}

\bigskip
\hrule
\bigskip

\noindent 2020 {\it Mathematics Subject Classification}: Primary 05A19; Secondary 11C20, 15B05.

\noindent \emph{Keywords:} Hessenberg--Toeplitz matrix, Trudi's formula, Schr\"{o}der number, Fine number, Catalan number, Schr\"{o}der path, Dyck path.

\bigskip
\hrule
\bigskip

\end{document}